\documentclass[12pt]{amsart}

\usepackage{amsmath,amssymb,amscd,amsfonts,amsthm,verbatim}
\usepackage[all,cmtip]{xy}
\usepackage{paralist}
\usepackage[mathscr]{eucal}
\usepackage{color}
\usepackage{pdfsync}
\usepackage[]{fontenc}
\usepackage{enumerate}

\usepackage[pdftex]{hyperref}
\hypersetup{citecolor=blue,linktocpage}

\usepackage[colorinlistoftodos]{todonotes}

\newtheorem{thm}{Theorem}[section]
\newtheorem{lem}[thm]{Lemma}
\newtheorem{assu-nota}[thm]{Assumption--Notation}

\theoremstyle{definition}
\newtheorem{defn}[thm]{Definition}
\newtheorem{rem}[thm]{Remark}

\newtheorem{assu}[thm]{Assumption}
\newtheorem{nota}[thm]{Notation}

\newcommand{\inv}{^{-1}}

\newcommand{\C}{\mathbb C}
\newcommand{\Z}{\mathbb Z}
\newcommand{\Q}{\mathbb Q}

\newcommand{\pp}{\mathbb P}

\newcommand{\sing}{_{\rm sing}}

\newcommand{\OO}{\mathcal O}

\newcommand{\cE}{\mathcal E}

\newcommand{\homc}{{\mathcal Hom}}
\newcommand{\ext}{{\mathcal Ext}}
\DeclareMathOperator{\Def}{Def}

\DeclareMathOperator{\Ext}{Ext}
\DeclareMathOperator{\Aut}{Aut}
\DeclareMathOperator{\Pic}{Pic}

\DeclareMathOperator{\im}{Im}

\DeclareMathOperator{\Jac}{Jac}

\numberwithin{equation}{section}

\title{Smoothing  semi-smooth stable Godeaux surfaces}
\author{Barbara Fantechi, Marco Franciosi and   Rita Pardini}

\begin{document}

\begin{abstract}
We show that all the semi-smooth stable complex Godeaux surfaces, classified in \cite{FPR18}, are smoothable, and that the moduli stack is smooth of the expected dimension 8 at the corresponding points. 
\smallskip

\noindent{\em 2020 Mathematics Subject Classification:} 14J10, 14D15, 14J29.

\noindent{\bf Keywords:} stable surface, semi-smooth surface, Godeaux surface, deformation,  smoothing.
\end{abstract}

\maketitle

\setcounter{tocdepth}{2}
\tableofcontents

\section{Introduction}

A {\em Godeaux surface} is (the canonical model of) a minimal complex surface of general type with $K^2=1$ and $h^1(\mathcal O)=h^2(\mathcal O)=0$. 
A {\em stable Godeaux surface} is a stable surface with the same numerical invariants; it is {\em semi-smooth} if it has only double crossings and pinch points as singularities, see \S \ref{ssec: semi-smooth}.

The algebraic fundamental group of a Godeaux surface is cyclic of order $m\le 5$ (\cite{Mi75}). Almost fifty years have passed since Reid's seminal paper  \cite{reid76} classifying Godeaux surfaces when $m\ge 3$,   but a classification in the simply connected case is still lacking,  in spite of much work on the subject (see Section \ref{Godeaux_history} for a recap of known facts on Godeaux surfaces and their moduli). In particular, the question of irreducibility has not been decided yet.

An approach to investigate dimension and singularities of the moduli is to construct non-canonical (i.e., having worse than canonical singularities) stable surfaces $X$ and show that they admit a smoothing:  the first such construction  can be found  in \cite[\S~7]{Lee07}, but many more examples are known nowadays (see,  for instance, \cite[\S~5]{Stern-Urzua}). When the singularities are non-isolated, this is technically more difficult and sometimes no smoothing exists, even for hypersurface singularities and $H^2(T_X)=0$ (see \cite{rollenske16}).

In  \cite{surfdef} we have obtained  deformation theoretical results that allow us to treat here the case of non-normal semi-smooth Godeaux surfaces; these are described explicitly in the classification of non-canonical stable Gorenstein surfaces in \cite{FPR18}. 

We verify for such surfaces the assumptions of Tziolas's formal smoothability criterion (Theorem 12.5 in \cite{tziolas}).

 \begin{thm}\label{thm: cohomological computation} 
Let $X$ be a  stable non-normal  semi-smooth Godeaux surface. Then 
\begin{itemize}
\item[{\rm(A)}] $\mathcal T^1_X$ is generated by global sections;
\item[{\rm(B)}] $H^1(X,\mathcal T^1_X)=0$;
\item[{\rm(C)}] $H^2(X, T_X)=0$.
\end{itemize}
\end{thm}

We prove Theorem \ref{thm: cohomological computation}  in Section \ref{main section}. 
The proof combines the explicit classification of the relevant surfaces as push-outs of their normalizations from \cite{FPR18} 
(cf.~Section \ref{sec: godeaux}) and the computation of $T_X$ and $\mathcal T^1_X$ for a semi-smooth variety $X$, again in terms of its construction as push-out, carried out in \cite{surfdef}. 
We find the proof of (C) (Section \ref{ssec: TX}) particularly interesting as it exploits the interplay between maps in cohomology and their geometrical interpretations. 

For stable surfaces formal smoothability is equivalent to geometric smoothability (see Section \ref{smoothability conditions} for details); thus the Theorem above has consequences on moduli:

\begin{thm}\label{thm: main}
Let  $X$ be a  non-normal  stable semi-smooth Godeaux surface. \begin{enumerate} 
\item The moduli stack of stable surfaces is nonsingular at $[X]$;
\item The general point of  the unique irreducible component $\overline M_X$ containing $[X]$  corresponds to a nonsingular surface;
\item $\overline M_X$ has (expected) dimension $8$.
\end{enumerate}
\end{thm}

\begin{rem} The fundamental groups of  semi-smooth non-normal Godeaux surfaces range over all the groups $\Z_m$, $m\le 5$;  by semicontinuity of the fundamental group in families (cf Proposition 4.5 of  \cite{FPR18}), it follows that for $m\ge 3$ the semi-smooth Godeaux surfaces with $\pi_1=\Z_m$ can be smoothed to Godeaux surfaces with the same fundamental group. For $m=1,2$ it is possible that the general surface in the same component of the moduli space  has larger fundamental group; however by Theorem \ref{thm: main} each semi-smooth non-normal Godeaux surface lies in exactly one  component, and this component has  the expected dimension 8. 
\end{rem}

We expect that the techniques developed in this paper can be extended to other singular stable surfaces, in particular the surfaces mentioned in Remark \ref{remark: non semi-smooth}.

\subsection{Godeaux surfaces and their moduli}\label{Godeaux_history}
We give here some context on Godeaux surfaces,  with the purpose of better  framing our results and methods.
\smallskip

Godeaux surfaces have been an object of intense study  over the last decades,  but a complete classification and  understanding of their moduli have not yet been achieved.
As we recalled above, the algebraic fundamental group $\pi_1^{\rm alg}$ of a Godeaux surface is cyclic of order $m\le 5$;   it is a folklore conjecture that   for each value of $m\le 5$ the connected component of the moduli of   Godeaux surfaces with $\pi_1^{\rm alg}=\Z_m$  is irreducible  and rational of dimension 8.  
This conjecture is known to be true for $m\ge 3$ by  \cite{reid76} (see also \cite{CU18}) and for $m=2$ by the recent preprint \cite{DR20}. 

Many ``sporadic'' examples  of Godeaux surfaces with trivial  $\pi_1^{\rm alg}$  are   known  (\cite{barlow2},
 \cite{CrGa94}, \cite{Lee07} and \cite{Stern-Urzua}) and an irreducible component of dimension 8  of the moduli space has very recently  been constructed in \cite{Schreyer-Stenger} by homological algebra methods, but  the the geometry of the moduli space is still mysterious. 
  For instance,  it is  known  (\cite[Theorem 0.31]{CP00}) that the local moduli space of 
 the surface  in \cite{CrGa94}   is smooth of dimension 8, but in general one has no clue to  which among the various examples belong to the same component of the moduli. 
 
 The most classical   approach to the construction (and eventually the classification) of Godeaux surfaces  with $\pi_1^{\rm alg}=\Z_m$  goes back to \cite{reid76} and  consists in writing down  the canonical ring of the universal cover of the  surface, keeping track of the $\Z_m$-action.  Clearly this method is ineffective when $m=1$,  and different techniques have been used in order to produce examples with trivial $\pi_1^{\rm alg}$.  
 One   method (cf.~\cite[\S~7]{Lee07}  and also \cite{Stern-Urzua}) consists in constructing  a normal surface with rational singularities and showing  that it admits a $\Q$-Gorenstein smoothing  to a simply connected Godeaux surface.  Namely, instead of constructing the Godeaux surface directly, one produces a surface in the boundary of the moduli space of stable surfaces with $K^2=1$ and $h^1(\OO)=h^2(\OO)=0$ and then prove its smoothability by deformation theoretical arguments. In this paper we apply  this approach to non-normal surfaces. 
 
 Our starting point is the   systematic analysis  of a  part of the boundary of the moduli  of stable Godeaux surfaces  carried out in \cite{FPR18}, where all the non-canonical   Gorenstein stable Godeaux surfaces   have been classified explicitly. 
 The question whether these surfaces actually belong to the closure of the moduli space of smooth Godeaux surfaces is partially answered in \cite{FPR18} and \cite{franciosi-rollenske},
 \cite{rollenske16}, but the smoothability of some of the non-normal examples is still to be decided.

 In  Theorem \ref{thm: main} we answer this question in the affirmative  for  semi-smooth surfaces; furthermore we show that the moduli stack is smooth of dimension 8      at the corresponding points,  as predicted by the folklore  conjecture above. In particular, the moduli stack is locally irreducible near these points: so while it is possible that  the closures of connected components of the moduli of canonical Godeaux with  different fundamental groups meet at the boundary, this does not happen at the points corresponding to semi-smooth stable  surfaces.

\subsection*{Acknowledgements}

We wish to thank Alessandro Nobile, Gian Pietro Pirola and Nikolaos Tziolas  for useful mathematical discussions. 

This article is based upon work supported by the National Science Foundation under
 Grant No. 1440140, while the first and third author were in residence at the Mathematical Sciences Research Institute in Berkeley, California, 
 during the spring semester of 2019.  
 This collaboration  started during the workshops ``Derived Algebraic Geometry and Birational Geometry and Moduli Spaces'' and
  ``Connections for Women: Derived Algebraic Geometry'',  January 2019, MSRI Berkeley. 

This project was partially supported by the projects PRIN  2015EYPTSB$\_$ 010 ``Geometry of Algebraic Varieties" and PRIN 
 2017SSNZAW$\_$004 ``Moduli Theory and Birational Classification"  of Italian MIUR.
 All authors are members of GNSAGA of INDAM.

\section{Smoothability conditions}\label{smoothability conditions}

The local analysis of the moduli stack of stable surfaces relies on the study of deformations of a stable surface.
We recall here the smoothability criterion we will use.

\begin{assu} In this section 
$X$ is  a proper, pure dimensional complex variety with complete intersection singularities (cf. \cite{surfdef}, Definition 2.2).
\end{assu}

We are interested in the existence of a geometric smoothing. 
\begin{defn}\label{geom smoothing} 
A {\em geometric smoothing} of $X$ is a Cartesian diagram 
\[
\begin{CD}
X @>{\iota}>>\mathcal X \\
@VVV @VV{\pi}V\\
c@>>> C.
\end{CD}
\]
where $C$ is a smooth irreducible curve, $c \in C$  is a closed point and $\pi$ is a flat and proper morphism, such that $\pi$ is generically smooth.
We say that $X $ is {\em geometrically smoothable} if it has a geometric smoothing.
 \end{defn} 

We denote by $T_X:=\homc(\Omega_X,\mathcal O_X)$  the  tangent sheaf of $X$ and by
$\mathcal T^1_X$ the sheaf $\ext^1(\Omega_X,\mathcal O_X)$.
The key theorem we use is the following result of Tziolas guaranteeing the existence of a {\em formal smoothing} (Definition 11.6 of \cite{tziolas}). 
 \begin{thm}\label{thm: tziolas} (\cite{tziolas}, Theorem 12.5). If the following conditions hold
\begin{itemize}
\item[{\rm(A)}] $\mathcal T^1_X$ is generated by global sections;
\item[{\rm(B)}] $H^1(X,\mathcal T^1_X)=0$;
\item[{\rm(C)}] $H^2(X, T_X)=0$.
\end{itemize}
Then X is formally smoothable, i.e. it admits a formal smoothing.
\end{thm}

Every geometric smoothing induces a formal smoothing, but the converse is in general not true. However, in our case the existence of the formal smoothing is sufficient, in view of the following result.

\begin{thm}(\cite{nobile}\label{thm: nobile})
If one of the following conditions holds:
\begin{enumerate}
\item $H^2(X,\OO_X)=0$,
\item either the dualizing sheaf $\omega_X$ or its dual $\omega_X^{\vee} $ is ample.
\end{enumerate}
Then $X$ is formally smoothable if and only if it is geometrically smoothable.
\end{thm}

\begin{rem}\label{rem: unobstructed}
 In general the assumptions (B) and (C) of Theorem \ref{thm: tziolas} imply that $X$ has unobstructed deformations.  By  \cite[Prop.~2.1.2.3]{illusie},  an obstruction space for deformations is given by $\Ext^2(\mathbb L_X,\mathcal O_X)$: since $X$  has complete intersection singularities and is reduced, the cotangent complex $\mathbb L_X$ of $X$ is equivalent to $\Omega_X$ in the derived category  (cf. also \cite[\S 2]{surfdef}). Thus $\Ext^2(\Omega_X,\mathcal O_X)$ is an obstruction space  for $X$. Since $\ext^2(\Omega_X,\mathcal O_X)=0$ because $X$ has complete intersection singularities, the result follows by the local to global spectral sequence of $\Ext$.
 \end{rem}

Recall that a proper surface  $X$ is {\em stable}  if it has semi-log canonical singularities (see \cite{KSB}, Section 4 and in particular Definition 4.17) and $K_X$ is ample as a $\Q$-Cartier divisor. 
\begin{rem} \label{rem: smooth}
By Remark \ref{rem: unobstructed},  if $X$ is a stable surface that satisfies conditions (B) and (C) of Theorem \ref{thm: tziolas}, then the moduli stack  of stable surfaces  is smooth at $[X]$ of dimension equal to $\dim \Ext^1(\Omega_X,\OO_X)$.
\end{rem}

\section{Semi-smooth stable Godeaux surfaces}\label{sec: godeaux}

\subsection{Semi-smooth surfaces} \label{ssec: semi-smooth}

We recall that a surface is {\em semi-smooth}  if it is  locally \'etale isomorphic to $\{u^2-v^2w=0\}$ (``double crossings points''); the singular points corresponding to the origin are called {\em pinch points} and  the remaining singular points are {\em double crossings points} (see e.g. Def.~4.1 in \cite{KSB}).

Semi-smooth surfaces have smooth normalization and the preimage  of the singular  locus via the normalization map is a smooth curve; more precisely, by \cite[Prop.~3.11]{surfdef} any quasi-projective semi-smooth surface $X$ can be obtained as follows. Let $\bar X$ be a smooth surface, $\bar Y\subset \bar X$ a smooth surface and  $g\colon \bar Y\to Y$ a double cover with $Y$ smooth: then  $X$ fits in the following push-out diagram:
\begin{equation}\label{eq: diag0}
\begin{CD}
\bar Y @>g>>Y\\
@V {\bar \jmath} VV @VV j V\\
\bar X@>f >>X
\end{CD}
\end{equation}
The maps  $\bar j$ and $j$ are closed embeddings and  $f$ is finite and birational (and so $f$ is the normalization map). The singular locus of $X$ is $Y$ and  the pinch points  are the images of the branch points of $g$.  One sometimes says that $X$ is obtained from $\bar X$ by gluing/pinching along $\bar Y$ via $g$ and writes  $X:=\bar X\sqcup_{\bar Y}Y$.
\medskip 

\subsection{Semi-smooth Godeaux surfaces} \label{ssec: godeaux}

 We call {\em stable Godeaux surface} a stable surface with $K^2=1$ and $h^1(\mathcal O)=h^2(\mathcal O)=0$.   
 The 
stable non-canonical (i.e.,  with  worse than canonical singularities) Gorenstein Godeaux surfaces  have been completely classfied in  \cite{FPR18}: in particular,    the  semi-smooth ones  are of type $(E_+)$, namely their normalization is the symmetric product of an elliptic curve. Here we  recall briefly  the  construction of these surfaces and set the notation. 

Fix an elliptic curve $E$,  let $P\in E$ be the origin and let  $\bar X=S^2E$ be the second  symmetric product of $E$.  The addition map $E\times E\to E$ induces the Albanese map  $\pi\colon \bar X \to E$. The map $\pi$ gives $\bar X$ the structure of a $\pp^1$-bundle over $E$: in fact we have $\bar X=\pp_E(\cE)$, where $\cE$ is the only non-trivial extension:
\begin{equation}\label{eq: extension}
0\to \OO_E\to \cE\to \OO_E(P)\to 0.
\end{equation}
We denote by $h$ the numerical equivalence  class of $\OO_{\pp_E(\cE)}(1)$ and by $F$ that  of a   fiber of $\pi$; the images in $\bar X$ of the ``coordinate curves'' $\{Q\}\times E$ are smooth curves of genus $1$ representing $h$. One has $h^2=hF=1$.

We let $\bar Y\subset \bar X$ be a  smooth curve of class $3h-F$. Since $K_{\bar X}$ is numerically equivalent to  $-2h+F$, we have  ${\bar Y}^2=3$, $K_{\bar X}\bar Y=-1$, so $\bar Y$ has genus $2$. We assume in addition that $\bar Y$ admits an involution $\iota$ with quotient a smooth curve $Y$ of genus 1 and we denote by $g\colon \bar Y\to Y$ the quotient map. The  existence and classification of such $(\bar Y,\iota)$ has been established in \cite{FPR18} and \cite{pd-elliptic}. So we can define  $X:=\bar X\sqcup_{\bar Y}Y$ as  the semi-smooth surface obtained by pinching $\bar X$ along $\bar Y$ via $g$.  By the Hurwitz formula the branch locus of $g$ consists of  two points, so $X$ has two pinch points. 

The line bundle $\omega_X=\OO_X(K_X)$ is ample, $K^2_X=1$ and $h^i(\OO_X)=0$ for $i>0$, namely $X$ is a stable Godeaux surface.    
By \cite{FPR18} all semi-smooth non-normal stable Godeaux surfaces arise in this way.
\begin{rem}\label{remark: non semi-smooth}
In fact, the construction  of type $(E_+)$ Godeaux surfaces in \cite{FPR18} includes also non semi-smooth surfaces, for special choices of the curve $E$. 
Assume the curve $E$ admits an endomorphism of degree 2. In this case   $\bar X$ contains a curve $\bar Y$ of class $3h-F$ that decomposes as $Y_1\cup Y_2$, where $Y_1$ has class $h$,  $Y_2$ has class $2h-F$, $Y_1$ and $Y_2$  meet transversely at one point $R$ and are both isomorphic to $E$. One can take $\iota$ to be an involution of $\bar Y$ that exchanges  $Y_1$ and $Y_2$ leaving $R$ fixed and set $X:=\bar X\sqcup_{\bar Y}Y$ also in this case. The surface $X$ is again a Gorenstein stable Godeaux surface but it has worse singularities, since it has a degenerate cusp at the image point of $R$. In this case our methods do not allow us to prove directly the smoothability of $X$. However, $X$ can be obtained as a limit of non-normal semi-smooth Godeaux surfaces, and so it is  smoothable too, but we don't know whether the moduli space is irreducible at $[X]$. \end{rem}

\section{Proof of Theorem   \ref{thm: cohomological computation} and  \ref{thm: main} }\label{main section}

\begin{nota}\label{nota: godeaux} 
We keep  the notation of \S \ref{ssec: godeaux}. In addition, we denote by $\bar y_i$, $i=1,2$ the fixed points of $\iota$ on $\bar Y$ and we set $y_i=g(\bar y_i)\in Y$. The action of $\iota$ induces a  decomposition into eigenspaces   $g_*\OO_{\bar Y}=\OO_Y\oplus L\inv$, where $L$ is a line bundle. The multiplication map of $g_*\OO_{\bar Y}$ induces an isomorphism  $L^{\otimes 2}\cong \OO_Y(B)$, where $B:=y_1+y_2$ is the branch locus of $g$. 
\end{nota}

\subsection{The sheaf $\mathcal T^1_X$}

The points $y_1$ and $y_2$ are the pinch points of $X$; the singular scheme  $X\sing$ of $X$ is supported on $Y$ but it is non reduced, since  it has embedded  points at   $y_1$ and $y_2$. Indeed,  locally in the \'etale topology $X$ is defined by the equation $h(u,v,w):= u^2-v^2w=0$, and  $X\sing$ is the scheme defined by the vanishing of $h$ and its derivatives, namely by   $u=vw=v^2=0$ (see \cite{surfdef}, \S 2.1).
The sheaf $\mathcal T_X^1$ is a line bundle on $X\sing$ whose restriction to $Y$ we denote by $\mathcal L$. So we have a short exact sequence:
\begin{equation}\label{eq: ext-seq}
0\to \mathcal K\to \mathcal T_X^1\to \mathcal L\to 0
\end{equation}
where $\mathcal K$ is isomorphic to $\C_{y_1}\oplus \C_{y_2}$.
Now  we show how  in this case  claims  ({\rm A}) and  ({\rm B}) of Theorem \ref{thm: cohomological computation} follow easily from Theorem  5.5 of  \cite{surfdef}:
\begin{lem}\label{lem: T1}
 \begin{enumerate}
\item  $h^0( \mathcal T_X^1)=7$ and $ \mathcal T_X^1$ is generated by global sections;
 \item $h^1( \mathcal T_X^1)=0$.
 \end{enumerate}
\end{lem}
\begin{proof}
One has  $H^i(\mathcal K)=0$ for $i>0$ because $\mathcal K$ has support of dimension zero, so taking global sections in  \eqref{eq: ext-seq} gives an exact sequence:
$$0\to H^0(\mathcal K)\to H^0(\mathcal T_X^1)\to H^0(\mathcal L)\to 0$$ 
and an isomorphism $H^1(\mathcal T_X^1)\cong H^1(\mathcal L)$.
It follows  that   $\mathcal T_X^1$ is generated by global sections if and only if $\mathcal L$ is.  By \cite{surfdef}, Theorem  5.5, there is an isomorphism $g^*\mathcal L\cong g^*L^{\otimes 2}\otimes N_{ \bar Y|\bar X}\otimes \iota^* N_{ \bar Y|\bar X}$; since $\deg N_{ \bar Y|\bar X}=3$, we have  $\deg g^*\mathcal L= 10 $ and therefore $\deg \mathcal L= 5$.  Since $Y$ is an  elliptic curve, then 
   $h^1(\mathcal L)=0$, 
   $\mathcal L$ is generated by global sections and   $h^0(\mathcal L)=5$, hence    $\mathcal T_X^1$ is generated by global section and  $h^0(\mathcal T_X^1)=5+2=7$, proving (i)
and (ii).

\end{proof}
\subsection{The sheaf $T_X$} \label{ssec: TX}
The proof of  claim (C) of Theorem \ref{thm: cohomological computation} also relies on the results of \cite{surfdef}, but is far more involved than the proof of   ({\rm A}) and  ({\rm B}). 

We start with some standard computations:

\begin{lem}\label{lem: cohomology}
\begin{enumerate}
\item $h^0(\bar X, T_{\bar X})=h^1(\bar X, T_{\bar X})=1$, $h^2(\bar X, T_{\bar X})=0$, 
\item $h^0(T_{\bar X}|_{\bar Y})=1$, $h^1(T_{\bar X}|_{\bar Y})=2$.
\end{enumerate}
\end{lem}
\begin{proof} (i) 
Taking the dual of the relative differentials sequence for the Albanese morphism $\pi\colon \bar X\to E$ 
$$0\to\pi^*\omega_E=\OO_{\bar X}\to \Omega^1_{\bar X}\to \omega_{\bar X|E}=\omega_{\bar X}\to 0, $$
one gets
\begin{equation}\label{eq: TX}
0\to \OO_{\bar X}(-K_{\bar X})\to T_{\bar X}\to \OO_{\bar X}\to 0
\end{equation}
We have $h^i(-K_{\bar X})=0$ for all $i$ (\cite{catanese-ciliberto} \S 2, (5)), hence the long cohomology sequence associated with \eqref{eq: TX} gives isomorphisms $H^i(\bar X, T_{\bar X})\cong H^i(\bar X,\OO_{\bar X})$ for every $i$, hence the claim.

\smallskip
(ii) Twisting \eqref{eq: TX} by $\OO_{\bar X}(-\bar Y)$ we get:
\begin{equation}
0\to \OO_{\bar X}(-C)\to T_{\bar X}(-\bar Y) \to \OO_{\bar X}(-\bar Y)\to 0,
\end{equation}\label{eq: TXY}
where $C$ is a divisor  in the numerical class  $h$ (recall that $K_{\bar X}$ is numerically equivalent to $-2h+F$). Since both  $C$ and $\bar Y$ are ample by \cite{Hartshorne}, Prop. V.2.21, by Kodaira vanishing the long exact sequence associated with \eqref{eq: TXY} gives $H^0(T_{\bar X}(-\bar Y))=H^1(T_{\bar X}(-\bar Y))=0$ and a short exact sequence 
$$0\to H^2(\bar X,\OO_{\bar X}(-C))\to H^2(\bar X, T_{\bar X}(-\bar Y))\to H^2(\bar X,\OO_{\bar X}(-\bar Y))\to 0.$$
By Riemann-Roch and Kodaira vanishing we have $h^2(\bar X,\OO_{\bar X}(-C))=\chi(\OO_{\bar X}(-C))=0$ and $h^2(\bar X,\OO_{\bar X}(-\bar Y))=\chi(\OO_{\bar X}(-\bar Y))=1$, and so $h^2(\bar X,T_{\bar X}(-\bar Y))=1$.

Consider now the sequence:
\begin{equation}
0\to T_{\bar X}(-\bar Y)\to T_{\bar X}\to T_{\bar X}|_{\bar Y}\to 0.
\end{equation}

By the previous computations,  taking cohomology one gets an  isomorphism $H^0(\bar X, T_{\bar X})\cong  H^0( \bar X, T_{\bar X}|_{\bar Y})$ and an exact sequence:
$$0\to H^1(\bar X, T_{\bar X})\to H^1( \bar X, T_{\bar X}|_{\bar Y})\to H^2(\bar X, T_{\bar X}(-\bar Y))\to 0.$$ 
Therefore by (i) we have   $ h^0( \bar X, T_{\bar X}|_{\bar Y})=1$ and  $ h^1( \bar X, T_{\bar X}|_{\bar Y})=h^1(\bar X, T_{\bar X})+h^2(\bar X, T_{\bar X}(-\bar Y))=1+1=2$.
\end{proof}

The next step is an analysis of $H^0(N_{\bar Y|\bar X})$. Since $\bar Y^2=3$  and $\bar Y$ has genus 2, by Riemann-Roch this is a $2$-dimensional vector space.  Consider the cohomology sequences:
\begin{equation}\label{eq: seq3}
0\to\OO_{\bar X} \to  \OO_{\bar X}(\bar Y)\to N_{\bar Y|\bar X}\to 0
\end{equation}
and
 \begin{equation}\label{eq: seq4}
 0\to T_{\bar Y}\to T_{\bar X}|_{\bar Y}\to N_{\bar Y|\bar X}\to 0
\end{equation}
and let $\gamma\colon H^0(N_{\bar Y|\bar X} )\to H^1(\OO_{\bar X})$ and $\delta\colon H^0(N_{\bar Y|\bar X})\to H^1(T_{\bar Y})$ be the coboundary maps induced by \eqref{eq: seq3}, \eqref{eq: seq4} respectively.

We have:
\begin{lem}\label{lem: delta-gamma}
\begin{enumerate}
\item $\ker \gamma$ and $\ker \delta$ have dimension 1;
\item $ H^0(N_{\bar Y|\bar X})= \ker \gamma \oplus \ker \delta$.
\end{enumerate}
\end{lem}
\begin{proof}

 By \cite{mumford}, Lecture 15,   there is a scheme $\mathcal P$ parametrizing the  curves of $\bar X$ algebraically equivalent to $\bar Y$ and $H^0(N_{\bar Y|\bar X})$ is canonically isomorphic to the  tangent space to $\mathcal P$ at the point $[\bar Y]$.  Denote by   $\Pic^{[\bar Y]}(\bar X)$ the connected component of $\Pic(\bar X)$ containing the class of $\OO_{\bar X}({\bar Y})$  
  and let $c\colon \mathcal P\to \Pic^{[\bar Y]}(\bar X)$ be the characteristic map,  that sends $[C]\in \mathcal P$ to the class of $\OO_{\bar X}(C)$.
  Since  the numerical class of $\bar Y$  is  $3h-F=h+K_{\bar X}$ and $h$ is ample, by Riemann-Roch and Kodaira vanishing we have $h^0(\OO_{\bar X}(C))=2$ for any curve algebraically equivalent to $\bar Y$, so the map $c$ gives $\mathcal P$  the structure of  a  $\pp^1$-bundle over the genus 1 curve $ \Pic^{[\bar Y]}(\bar X)$.  
  The  differential of $c$ at $[\bar Y]$ is $\gamma$ and  the long cohomology sequence coming from \eqref{eq: seq3} shows that $\ker \gamma$ is the image of $H^0(\OO_{\bar X}(\bar Y))\to H^0(N_{\bar Y| \bar X})$ and has dimension 1. 
  
The diagonal action of $E$ on $E\times E$ by translation descends to an  action on $\bar X=S^2E$. We denote by $g_P$ the automorphism of $\bar X$ induced by translation by a point  $P\in E$; the automorphism  $g_P$ acts on the curves in the numerical class of $h$ as twisting by $P$, where we regard $P$ as an element of $E=\Pic^0(\bar X)$.   Since   $\bar Y$  is numerically equivalent to $h+K_{\bar X}$, we have  
 that $g_P^*\bar Y$ is in the linear system $|\bar Y-P|$.  So  $S:=\{[g_P^*\bar Y]\ | \ P\in E\}\subset \mathcal P$ is a section of the $\pp^1$-bundle $\mathcal P$ and its tangent space $W$ at $[\bar Y]$ is a 1-dimensional subspace of $H^0(N_{\bar Y|\bar X})$ that is mapped  isomorphically to $H^1(\OO_{\bar X})$. We claim that $W=\ker \delta$. 

Indeed, the elements of $H^0(N_{\bar Y|\bar X})$ are the first order deformations of $\bar Y\subset \bar X$ and they are mapped by $\delta$ to the corresponding deformation of $\bar Y$. Since $S$ gives a trivial deformation of $\bar Y$, it is clear that $W$ is contained in $\ker \delta$. 
To finish the proof it is enough to observe that $\ker \delta$ has dimension 1: this follows from the long exact sequence associated with  \eqref{eq: seq4}, since 
$h^1(N_{\bar Y|\bar X})=0$ because  $\bar Y^2=3$ and $\bar Y$ has genus 2, $h^1(T_{\bar Y})=3$ and, by Lemma \ref{lem: cohomology}, $h^1(T_{\bar X|\bar Y})=2$. \end{proof}

The next step is an analysis of $H^0(K_{\bar Y})$. We start with a couple of general remarks. 
\begin{rem}\label{rem: jac1} Let $h\colon C_1\to C_2$ a finite morphism of curves of positive genus and write $J_i:=\Jac (C_i)$, $i=1,2$.  Choosing base points $x_1\in C_1$ and $x_2:=h(x_1)\in C_2$, we have a  commutative diagram
$$
\begin{CD}
C_1 @>a_1>>J_1\\
@V {h} VV @VV{h_* }V\\
C_2@> > a_2>J_2
\end{CD}
$$
where $a_1$ and $a_2$  are the Abel-Jacobi maps with base points $x_1,x_2$, respectively, and $h_*$ is the morphism of abelian varieties induced by $h$.  The differential of $h_*$ at the origin is the transpose of the pull-back map $h^*\colon H^0(K_{C_2})\to H^0(K_{C_1})$, so the tangent space to $\ker h_*$ at the origin is $(h^*(H^0(K_{C_2}))^{\perp}$.
\end{rem}
\begin{rem}\label{rem: jac2}
In the situation of Remark \ref{rem: jac1}, assume in addition that $h$ does not factor through a non-trivial \'etale cover of $C_2$. This happens,  for instance, if $\deg h$ is a prime and $h$ is not \'etale.  Then the morphism  $h^*\colon J_2= \Pic^0(C_2)\to J_1=\Pic^0(C_1)$ is injective and therefore the kernel $A$ of the dual morphism 
$h_*\colon J_1\to J_2$ is  connected.  In particular, if $C_2$ has genus 1, then $A$ is a connected divisor with $A\cdot a_1(C_1)=\deg h$. 
\end{rem}

We now apply the previous remarks in our situation.
Under the action of the involution  $\iota$ induced by the double cover $g\colon \bar Y\to Y$ the vector space $V:=H^0(K_{\bar Y})$  splits as  the direct sum  $V=V^+\oplus V^-$ of  $1$-dimensional  eigenspaces, with  $V^+=g^*H^0(K_Y)$.  
Denote by $\sigma$  the hyperelliptic involution of $\bar Y$: $\sigma$ and $\iota$ generate a group isomorphic to $ \Z_2^2$ and $\sigma$ acts on $V$ as multiplication by $-1$, so $V^-$  is invariant under the action of $\iota':=\iota\circ \sigma$. So, if   $g'\colon \bar Y\to Y':=\bar Y/\iota'$ denotes   the quotient map, the curve $Y'$ has genus 1 and  
$V^-={g'}^*H^0(K_{Y'})$.

Set $Z:=V^{\vee}$; since $\iota$ acts trivially on $H^1(K_{\bar Y})$, by Serre duality  $Z^-:=(V^+)^{\perp}$, and $Z^+:= (V^-)^{\perp}$. The space $Z$ contains also  the $1$-dimensional  subspace $W:=(p^*H^0(K_E))^{\perp}$ where $p\colon \bar Y\to E$ is  the degree 3 morphism induced by  the Albanese map of $\bar X$. 

 We have the following: 
\begin{lem}\label{lem: WZ}
One has
$$Z^+\cap W=Z^-\cap W=0.$$
\end{lem}
\begin{proof} Set $J:=\Jac(\bar Y)$ and denote by $\Theta\subset J$ the image of $\bar Y$ via the Abel-Jacobi map. 
By  definition $Z$ is the  tangent space to $J$ at the origin; by Remark \ref{rem: jac1},    $Z^-$ is the tangent space at the origin to the kernel $D$ of $g_*\colon J\to \Jac(Y)$,  $Z^+$ is the tangent space at the origin to the kernel $D'$ of $g'_*\colon J\to \Jac(Y')$, and $W$ is the tangent space to the kernel $E'$  of $p_*\colon J\to \Jac(E)$.
By Remark \ref{rem: jac2} $D$, $D'$  and $E'$ are connected and satisfy $\Theta\cdot D=\Theta\cdot D'=2$ and $\Theta\cdot E'=3$. 
 Summing up,  the 3 abelian  subvarieties $D,D'$ and $E'$ are distinct;   since an abelian subvariety is  determined by its tangent space at the origin,  $Z^+$, $Z^-$ and $W$ are pairwise distinct.
\end{proof}
The next result is the key ingredient of  the proof of fact (C):
\begin{lem}  \label{lem: key}
We have: 
$$\im \delta \cap H^1(T_{\bar Y})^+=0.$$
\end{lem}
\begin{proof}
The space $H^1(T_{\bar Y})$ is Serre dual to $H^0(2K_{\bar Y})$. Since $\bar Y$ has genus 2, the multiplication map $\mu \colon H^0(K_{\bar Y})\otimes H^0(K_{\bar Y})\to H^0(2K_{\bar Y})$  induces an isomorphism $\rho\colon S^2H^0(K_{\bar Y})\to H^0(2K_{\bar Y})$. Via these identifications, we have $H^1(T_{\bar Y})^+=\left(Z^+\otimes Z^+\right)\oplus \left(Z^-\otimes Z^-\right)$. Also, there is an isomorphism  $H^1(T_J)\cong Z\otimes Z$ and the dual map $^t\mu\colon H^1(T_{\bar Y})\to Z\otimes Z$ is the differential at $[\bar Y]$ of the Torelli map, sending a curve of genus 2 to its  Jacobian. 

To simplify the notation in what follows we set $\psi:=p_*\colon J\to \Jac(E)\cong E$.  The differentials sequence  $0\to T_{J/E}\to T_J\to \psi^*T_E\to 0$ can be rewritten more explicitly as: 
\begin{equation}
0\to W\otimes \OO_J\to Z\otimes \OO_J\to W'\otimes\OO_J\to 0,
\end{equation}
where $W'=H^0(K_E)^{\vee}$ is the tangent space to $E$ at the origin. 
Following the notation of \cite[\S~3.4.2]{sernesi},  we denote by $\Def_{\psi/E}$  the deformations with fixed target of the map $\psi \colon J\to E$. By {\em ibid.}, Thm.~3.4.8 and Lem.~3.4.7, (iv), the tangent space to $\Def_{\psi/E}$ is $H^1(T_{J/E})=W\otimes Z$; moreover the map $H^1(T_{J/E})\to H^1(T_J)$ is clearly an inclusion. 

By Lemma \ref{lem: delta-gamma},  the image of $\delta$ is $\delta(\ker \gamma)$, namely it is generated by the first order deformation $\xi$ of $\bar Y$ obtained by letting $\bar Y$ vary in the linear pencil $|\bar Y|$ of $\bar X$.   The element $^t\mu(\xi)$ is the corresponding first order deformation of $J$ and, since $\xi$ induces  a first order deformation  of $\psi$ with fixed target,  by the above discussion it lies in $H^1(T_{J/E})=W\otimes Z$. Using Lemma \ref{lem: WZ}, it is  an easy linear algebra exercise to show that the subspaces $H^1(T_{\bar Y})^+=\left(Z^+\otimes Z^+\right)\oplus \left(Z^-\otimes Z^-\right)$ and $H^1(T_{J/E})=W\otimes Z$ of $H^1(T_J)=Z\otimes Z$ intersect only in 0.

\end{proof}
\subsection{Conclusion}\label{ssec: conclusion}
We are finally ready complete the proofs.
\begin{proof}[Proof of Theorem \ref{thm: cohomological computation}]
Claims (A) and (B) are proven in Lemma \ref{lem: cohomology} and Lemma \ref{lem: T1}, so we only have to prove claim (C).
We recall first  some facts from \cite{surfdef}. By  \cite[Thm.~5.1]{surfdef}   there is  a natural injective map $\alpha\colon  T_X\to f_*T_{\bar X}$ which is an isomorphism on  the smooth locus of $X$. Let $\mathcal G$ be the sheaf  defined by the short exact sequence:
\begin{equation}\label{eq: seq1}
0\to T_X\overset{\alpha}{\to} f_*T_{\bar X}\to \mathcal G\to 0. 
\end{equation} 
The map  $\alpha$ is an isomorphism on the smooth locus of $X$, so the sheaf $\mathcal G$ is supported on $Y$. By the same Theorem,  there is an exact sequence 
\begin{equation}\label{eq: seq2}
0\to (g_*T_{\bar  Y})^+\to g_*T_{\bar  X}|_{\bar  Y}\to \mathcal G\to 0.
\end{equation}
Since $H^2(f_*T_{\bar X})=H^2(\bar X, T_{\bar X})=0$ by Lemma \ref{lem: cohomology}, by \eqref{eq: seq1} it is enough to show that $h^1(\mathcal G)=0$ or, equivalently, 
 that the map  $j\colon H^1((g_*T_{\bar  Y})^+)  \to H^1( g_*T_{\bar  X}|_{\bar  Y})$ is surjective. Since $g$ is a finite map, we can make identifications  $H^1((g_*T_{\bar  Y})^+)\cong H^1(T_{\bar Y})^+$ and $H^1( g_*T_{\bar  X}|_{\bar  Y})\cong H^1(T_{\bar X}|_{\bar Y})$ and work on $\bar Y$.
  
 Taking cohomology in \eqref{eq:  seq4}
 we get:
 \begin{eqnarray}\label{eq: long}
 0\to H^0(T_{\bar X}|_{\bar Y})\to H^0(N_{\bar Y|\bar X})\overset{\delta}{\to} H^1(T_{\bar Y})\overset{j_0}{\to}  H^1(T_{\bar X}|_{\bar Y})\to 0.
  \end{eqnarray}
 So the map $j$ is just the restriction to $H^1(T_{\bar Y})^+$ of the map $j_0$ in \eqref{eq: long}; since both $H^1(T_{\bar Y})^+$ and $H^1(T_{\bar X}|_{\bar Y})$  have dimension two,  $j$ is surjective iff it is an isomorphism iff the kernel of $j_0$, that is the  image of $\delta$, intersects $H^1(T_{\bar Y})^+$ only in zero. This last statement is precisely the content of Lemma \ref{lem: key}, so fact (C) is proven. 
\end{proof}

\begin{proof}[Proof of Theorem \ref{thm: main}]
By Theorem \ref{thm: cohomological computation} the assumptions of Theorem \ref{thm: tziolas} are satisfied, so $X$ is formally smoothable. Since $\omega_X$ is ample, or since $H^2(\OO_X)=0$,  Theorem \ref{thm: nobile} applies and therefore  $X$ is geometrically smoothable and  claim (ii) is proven.  Furthermore,  by Remark \ref{rem: smooth} the stack $\overline M_X$ is smooth at $[X]$ of dimension equal to $\dim \Ext^1(\Omega_X,\OO_X)$. 

To complete the proof we need to show  that $\Ext^1(\Omega_X,\OO_X)$ has dimension 8. 
Since $H^2(T_X)=0$, by  the local-to-global exact sequence for $\Ext$ 
we have 
 \begin{gather}
0\to H^1(T_X)\to \Ext^1(\Omega_X,\OO_X)\to H^0(\mathcal T^1_X) \to 0 \nonumber 
\end{gather}
hence 
 $\dim \Ext^1(\Omega_X,\OO_X)=h^1(T_X)+h^0(\mathcal T^1_X)=h^1(T_X)+7$, where the last equality follows by Lemma \ref{lem: T1}.  So we have to prove $h^1(T_X)=1$. Again by the vanishing of $H^2(T_X)$  we have $h^1(T_X)=h^0(T_X)-\chi(T_X)=-\chi(T_X)$, since $H^0(T_X)$ is the tangent space at the origin of $\Aut(X)$ and $\Aut(X)$ is finite because $X$ is stable
 (cf. \cite{BHPS}, Lemma 2.5).
 Finally, sequences \eqref{eq: seq2} and \eqref{eq: seq1} give  
\begin{gather*} \chi(T_X)=\chi(T_{\bar X})-\chi(\mathcal G)=\\
\chi(T_{\bar X})-\chi(T_{\bar X}|_{\bar Y})+\chi((g_*T_{\bar Y})^+)=0-(-1)-2=-1
\end{gather*}
where the last equality follows by Lemma \ref{lem: cohomology} and by observing that $\chi((g_*T_{\bar Y})^+)=-\dim H^1(T_{\bar Y})^+=-2$ (see the proof of Lemma \ref{lem: key}).
\end{proof}

      \hfill\break
Barbara Fantechi \\
SISSA  \\
 Via Bonomea 265, I-34136 Trieste (Italy)\\
{\tt fantechi@sissa.it}

  \vspace{.2cm}
   \hfill\break
Marco Franciosi\\
Dipartimento di Matematica, Universit\`a di Pisa\\
Largo B.~Pontecorvo 5,  I-56127 Pisa (Italy)\\
{\tt marco.franciosi@unipi.it}

\vspace{.2cm}

 \hfill\break
Rita Pardini\\
Dipartimento di Matematica, Universit\`a di Pisa\\
Largo B.~Pontecorvo 5,  I-56127 Pisa (Italy)\\
{\tt rita.pardini@unipi.it}

     \end{document}